\documentclass{amsart}
\usepackage{graphicx} 
\usepackage{times,colordvi,amsmath,epsfig,float,multicol,subfigure} 


\title[On Various Types of Shadowing for Geometric Lorenz Flows]
      {On Various Types of Shadowing for Geometric Lorenz Flows}

\author[A. Arbieto]{A. Arbieto}
\address{Instituto de Matem\'atica, Universidade Federal do Rio de Janeiro, P. O. Box 68530, 21945-970 Rio de Janeiro, Brazil.}
\email{arbieto@im.ufrj.br}

\author[J.E. Reis]{J.E. Reis}
\address{Instituto de Matem\'atica, Universidade Federal do Rio de Janeiro, P. O. Box 68530, 21945-970 Rio de Janeiro, Brazil.}
\email{$\textmd{joao\_eduardo\_reis@im.ufrj.br}$}

\author[R. Ribeiro]{R. Ribeiro}
\address{Instituto de Matem\'atica, Universidade Federal do Rio de Janeiro, P. O. Box 68530, 21945-970 Rio de Janeiro, Brazil.}
\email{raquelribeiro@im.ufrj.br}

\thanks{The first author was partially supported by CNPq, FAPERJ and PRONEX/DS from
Brazil.}
\thanks{The second author was supported by CAPES}
\thanks{The third author was supported by CNPq.}

\usepackage[usenames]{color}
\newtheorem{theorem}{Theorem}

\newtheorem{lemma}[theorem]{Lemma}
\newtheorem{proposition}[theorem]{Proposition}

\theoremstyle{definition}
\newtheorem{definition}[theorem]{Definition}
\newtheorem{remark}[theorem]{Remark}

\newcommand{\F}{\mathcal{F}}

\begin{document}

\begin{abstract}
We show that Lorenz flows have neither limit shadowing property
nor average shadowing property nor the asymptotic average
shadowing property where the reparametrizations related to these
concepts relies on the set of increasing homeomorphisms with
bounded variation.
\end{abstract}

\maketitle

\section{Introduction}

The shadowing property is a dynamical property that plays a key
role in the study of the stability of the dynamics. This property
is found in hyperbolic dynamics and it was used with success to
prove their stability, see for instance \cite{shub}. Roughly
speaking, it allows us to trace a set of point which looks like an
orbit, but with errors, by a true orbit. For practical
applications, we can suppose that $\varphi$ is viewed as the orbit
realized in numerical calculation by computer, or in physical
experiments, thus it could have errors. Then shadowing property
allow us to ``correct'' this errors, finding a true evolution
which nicely approximates $\varphi$. Thus, to decide which systems
possess the shadowing property is an important problem in
dynamics.

The geometric Lorenz model is an important example in the theory
of dynamical systems, it was inspired by the equations found by
Lorenz related to a model of fluid convection \cite{rr9}.
Moreover, it is one of the most famous examples, since it is often
related with the notion of chaos. It was studied in the initial
stages by Guckenheimer-Williams \cite{rr4}, \cite{rr6},
\cite{rr5}, Afraimovich-Bykov-Shil'nikov \cite{rr7}  and
Yorke-Yorke \cite{rr8}. Moreover it is an attractor which is
transitive and contains both regular orbits and singularities. As
we mentioned before, the hyperbolic dynamical systems possess the
shadowing property. However, these Lorenz systems are not
hyperbolic, since they have singularities approximated by regular
orbits. Even so, they have some robust properties which are also
shared by hyperbolic dynamics.

It is natural then to ask if these Lorenz systems have the
shadowing property. Komuro \cite{rr10} showed  that geometric
Lorenz flows do not satisfy the (parameter-fixed) shadowing
property excepted in very restricted cases. Even so, in
\cite{rr11} it was shown that the geometric Lorenz attractors have
the parameter-shifted shadowing property. However, this notion is
very technical. So, we could ask if these systems have some
shadowing-type properties which are more easy to check.

Related to this, many properties were suggested and studied by
several authors. As a kind of generalization of the shadowing
property, Blank \cite{rr12} introduced the notion of the average
shadowing property in the study of chaotic dynamical systems.
Essentially it allows great errors, but they must be compensated
with small errors.

In the other hand, Eirola \textit{et al} \cite{rr13} posed the
notion of the limit-shadowing property. From the numerical point
of view this property on a dynamical system $X$ means that if we
apply a numerical method of approximation to $X$ with ``improving
accuracy" so that one step errors tend to zero as time goes to
infinity then the numerically obtained trajectories tend to real
ones. Such situations arise, for example, when one is not so
interested on the initial (transient) behavior of trajectories but
wants to reach areas where ``interesting things" happen (e.g.
attractors) and then improve accuracy. In the sequence, Rongbao Gu
\cite{rr14} introduced the notion of the asymptotic average
shadowing property for flows. This is a certain generalization of
the limit-shadowing property in random dynamical systems.

It could be checked that these weaker shadowing properties are
present in hyperbolic dynamics. Thus, following Komuro, a natural
question is to decide if the Lorenz systems has some of these
weaker shadowing properties. We remark also, as it is well known,
that the analysis of shadowing on flows becomes more complicated
than the analysis for diffeomorphims, due to the presence of
reparametrizations of the systems on those concepts.

The purpose of this paper is to seek sufficient conditions over
the Lorenz map which implies these kinds of shadowing:
$\Delta$-asymptotic average shadowing property (abbrev.
$\Delta$-AASP), $\Delta$-limit shadowing property (abbrev.
$\Delta$-LSP) and $\Delta$-average shadowing property (abbrev.
$\Delta$-ASP). The constant $\Delta$ is an upper bound to the
variation of the respective reparametrizations (see \S
\ref{zz61}).


Let $(\Lambda , \varphi )$ be a geometric Lorenz flow with
Poincar\'{e} map $P:S^*\to S$ (see definitions on \S \,
\ref{zz61}). Let $f: \F^* \to \F$ be the map associated to the
foliation $\F$ on $S$ and $L^+$ and $L^-$ be the lateral leaves of
$\F$. Our main result is the following:

\begin{theorem}\label{zz9}
Lorenz flows $(\Lambda,\varphi)$ satisfying $f(L^+)\neq L^+$ or
$f(L^-)\neq L^-$ have neither $\Delta$-ASP, $\Delta$-LSP nor
$\Delta$-AASP for any $\Delta\geq0$.
\end{theorem}

This paper is organized as follows: in \S \, \ref{zz61} and \S \,
\ref{zz116} we give the precise definitions of the objects used in
the statement above. In \S \, \ref{zz52} we prove Theorem
\ref{zz9}.

\section{Various types of Shadowing}\label{zz61}

There are several types of shadowing in the literature. In this
section we define the ones which will be worked in this paper.

Let
\begin{equation*}
    Rep=\{g:\mathbb{R}\to \mathbb{R}: \hspace{.166cm}g\hspace{.166cm}
    \textmd{is}\hspace{.166cm} \textmd{a}\hspace{.166cm}
    \textmd{monotone}\hspace{.166cm} \textmd{increasing}\hspace{.166cm}
    \textmd{homeomorphism} \hspace{.166cm}\textmd{with} \hspace{.166cm}g(0)=0 \}.
\end{equation*}
Fixed $\Delta \in R^+$, define
\begin{equation*}
    Rep(\Delta)=\left\{g\in Rep: \left|\frac{g(s)-g(t)}{s-t}-1\right|\leq \Delta, \hspace{.166cm}
    \textmd{for}  \hspace{.166cm}\textmd{every}  \hspace{.166cm}s,t\in \mathbb{R}\right\}.
\end{equation*}


A sequence $(x_i,t_i)_{i\in \mathbb{Z}}$ is a
$\delta$-\textit{average-pseudo orbit} of $\varphi$ if $t_i\geq 1$
$\forall$ $i\in \mathbb{Z}$ and there is a positive integer $N$
such that for any $n\geq N$ and $k\in \mathbb{Z}$ we have
\begin{equation*}\label{zz78}
    \frac{1}{n}\sum_{i=1}^nd(\varphi({t_{i+k}},x_{i+k}),x_{i+k+1})<\delta.
\end{equation*}
A $\delta$-average-pseudo orbit, $(x_i,t_i)_{i\in \mathbb{Z}}$, of
$\varphi$ is \emph{$\Delta$-positively $\epsilon$-shadowed in
average} by the orbit of $\varphi$ through
 $x$, if there exists $h \in Rep(\Delta)$ such that
\begin{equation}\label{zz44}
    \limsup_{n\to\infty}\frac{1}{n}\sum_{i=1}^n\int_{s_i}^{s_{i+1}}d(\varphi({h(t)},x), \varphi({t-s_i},x_i))dt<\epsilon,
\end{equation}
where $s_0=0$ and $s_n=\sum_{i=0}^{n-1}t_i$, $n\in \mathbb{N}$. It
is \emph{$\Delta$-negatively $\epsilon$-shadowed in average} by
the orbit of $\varphi$ through $x$ if there is $\tilde{h}\in
Rep(\Delta)$ for which the limit \eqref{zz44} is true when
replacing $h$ by $\tilde{h}$ and the limits of integration by
$-s_{-i}$ and $-s_{-i+1}$ (in this case
$s_{-n}=\sum_{i=-n}^{-1}t_i$).

\begin{definition}\label{zz46}
The flow $\varphi$ has the \emph{$\Delta$-average shadowing
property} (abbrev. $\Delta$-ASP) if for any $\epsilon>0$ there
exists $\delta>0$ such that any $\delta$-average-pseudo orbit of
$\varphi$ is both $\Delta$-positively (negatively)
$\epsilon$-shadowed in average by some orbit of $\varphi$.
\end{definition}

A sequence $(x_i,t_i)_{i\in \mathbb{Z}}$ is a \textit{limit-pseudo
orbit} of $\varphi$ if $t_i>1$ $\forall$ $i\in \mathbb{Z}$ and
\begin{equation*}\label{zz77}
    \lim_{|i|\rightarrow\infty}d(\varphi(t_i,x_i),x_{i+1})=0.
\end{equation*}
A limit-pseudo orbit, $(x_i,t_i)_{i\in \mathbb{Z}}$, of $\varphi$
is $\Delta$-\textit{positively shadowed in limit} by an orbit of
$\varphi$ through $x$ if there is $h \in Rep(\Delta)$ such that
\begin{equation*}\label{zz49}
    \lim_{i\to\infty}\int_{s_i}^{s_{i+1}}d(\varphi({h(t)},x), \varphi(t-{s_i},
    x_i))dt=0.
\end{equation*}
Analogously as before we define when a limit-pseudo orbit is
$\Delta$-\textit{negatively shadowed in limit} by an orbit.

\begin{definition}\label{zz47}
The flow $\varphi$ has the $\Delta$-\emph{limit shadowing
property} (abbrev. $\Delta$-LSP) if every limit-pseudo orbit is
both $\Delta$-positively (negatively) shadowed in limit by an
orbit of $\varphi$.
\end{definition}

A sequence $(x_i,t_i)_{i\in \mathbb{Z}}$ is an \textit{asymptotic
average-pseudo orbit} of $\varphi$ if $t_i\geq 1$ $\forall$ $i\in
\mathbb{Z}$ and
\begin{equation*}
    \lim_{n\to \infty}\frac{1}{n}\sum_{i=-n}^nd(\varphi({t_i},x_i),x_{i+1})=0.
\end{equation*}
An asymptotic average-pseudo orbit, $(x_i,t_i)_{i\in \mathbb{Z}}$,
of $\varphi$ is $\Delta$-\textit{positively asymptotically
shadowed in average} by an orbit of $\varphi$ through $x$ if there
exists $h\in Rep(\Delta)$ such that
\begin{equation*}\label{zz45}
    \lim_{n\to\infty}\frac{1}{n}\sum_{i=0}^n\int_{s_i}^{s_{i+1}}d(\varphi({h(t)},x), \varphi({t-s_i},x_i))dt=0.
\end{equation*}
Similarly we define when an asymptotic average-pseudo orbit is
$\Delta$-\textit{negatively asymptotic shadowed in average} by an
orbit.

\begin{definition}\label{zz48}
The flow $\varphi$ has the $\Delta$-\emph{asymptotic average
shadowing property} (abbrev. $\Delta$-AASP) if every asymptotic
average-pseudo orbit is both $\Delta$-positively (negatively)
asymptotically shadowed in average by an orbit of $\varphi$.
\end{definition}

\section{Geometric Lorenz Flows}\label{zz116}

\subsection{Construction}


Let $S^3 = \mathbb{R}^3 \cup \{ \infty \}$ be the $3$-sphere. The
Geometric Lorenz Attractor is an attractor set in $S^3$ of a flow
denoted by $Y_t$ that we are about to describe. This attractor
has, as an isolating block, a solid bitorus $U$ in $\mathbb{R}^3$
such that the flow $Y_t$ is inwardly transverse to the boundary of
$U$. In $S^3 \backslash U$ the flow $Y_t$ has three saddle type
hyperbolic singularities with stable complex eigenvalues and a
source in $\{\infty\}$. Define

$$\Lambda = \bigcap_{t\geq0}Y_t(U)$$

\noindent as the maximal $Y_t$-invariant set in $U$. The set
$\Lambda$ called the \textit{Geometric Lorenz Attractor}. See
Figure \ref{zz118}. This geometric model is motivated by the
Lorenz field
\begin{equation}\label{zz131}
    X(x, y, z) = (-ax + ay, rx - y - xz, xy - bz)\; ,  \;\;\;\;\ a, r, b > 0
\end{equation}


\begin{figure}[h]
\begin{center}
\subfigure[\label{zz117}Behavior near the origin.]{
\includegraphics[scale=0.3855]{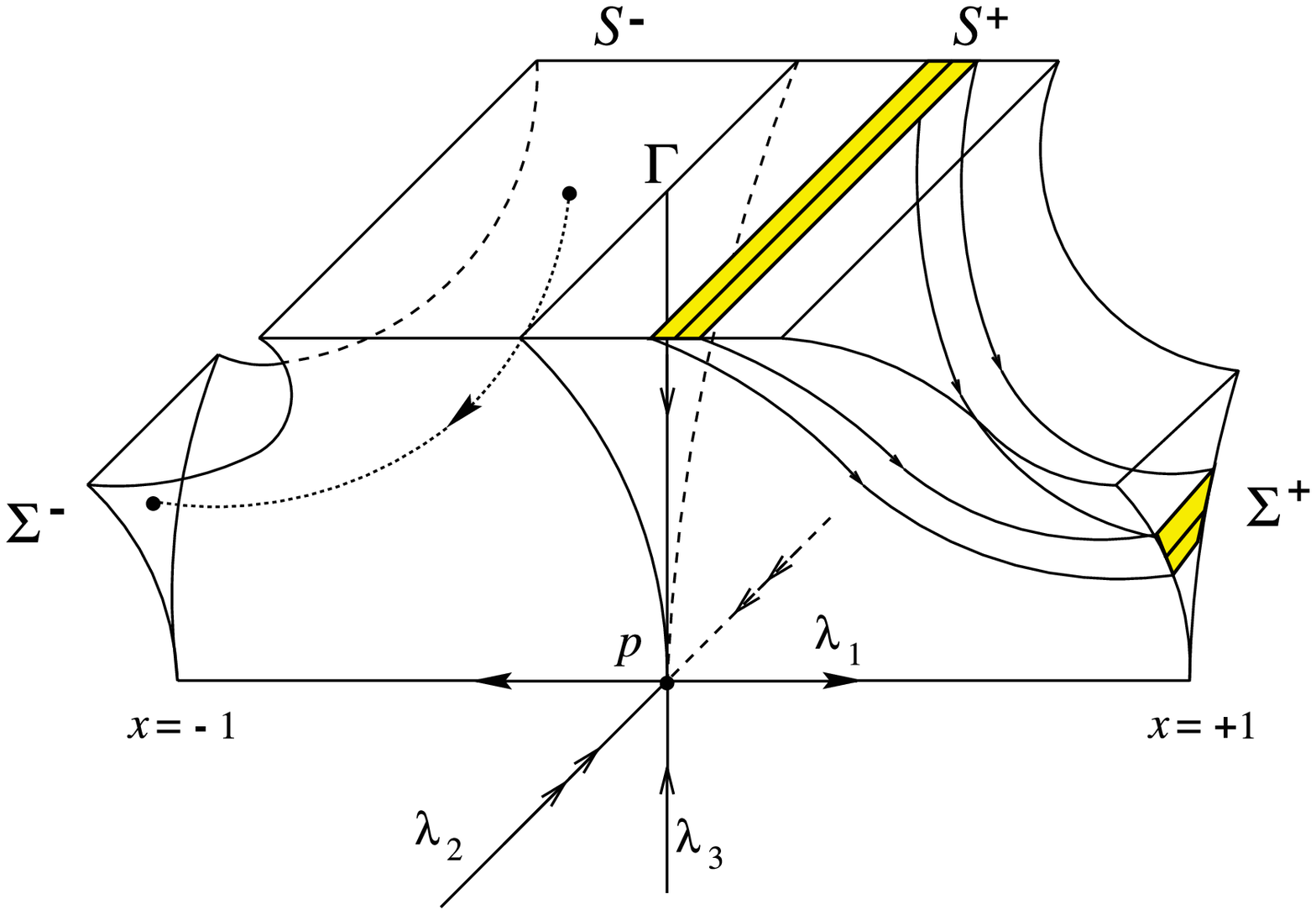}}
\subfigure[\label{zz118} The flow takes $\Sigma^\pm$ to $S$.]{
\includegraphics[scale=0.4255]{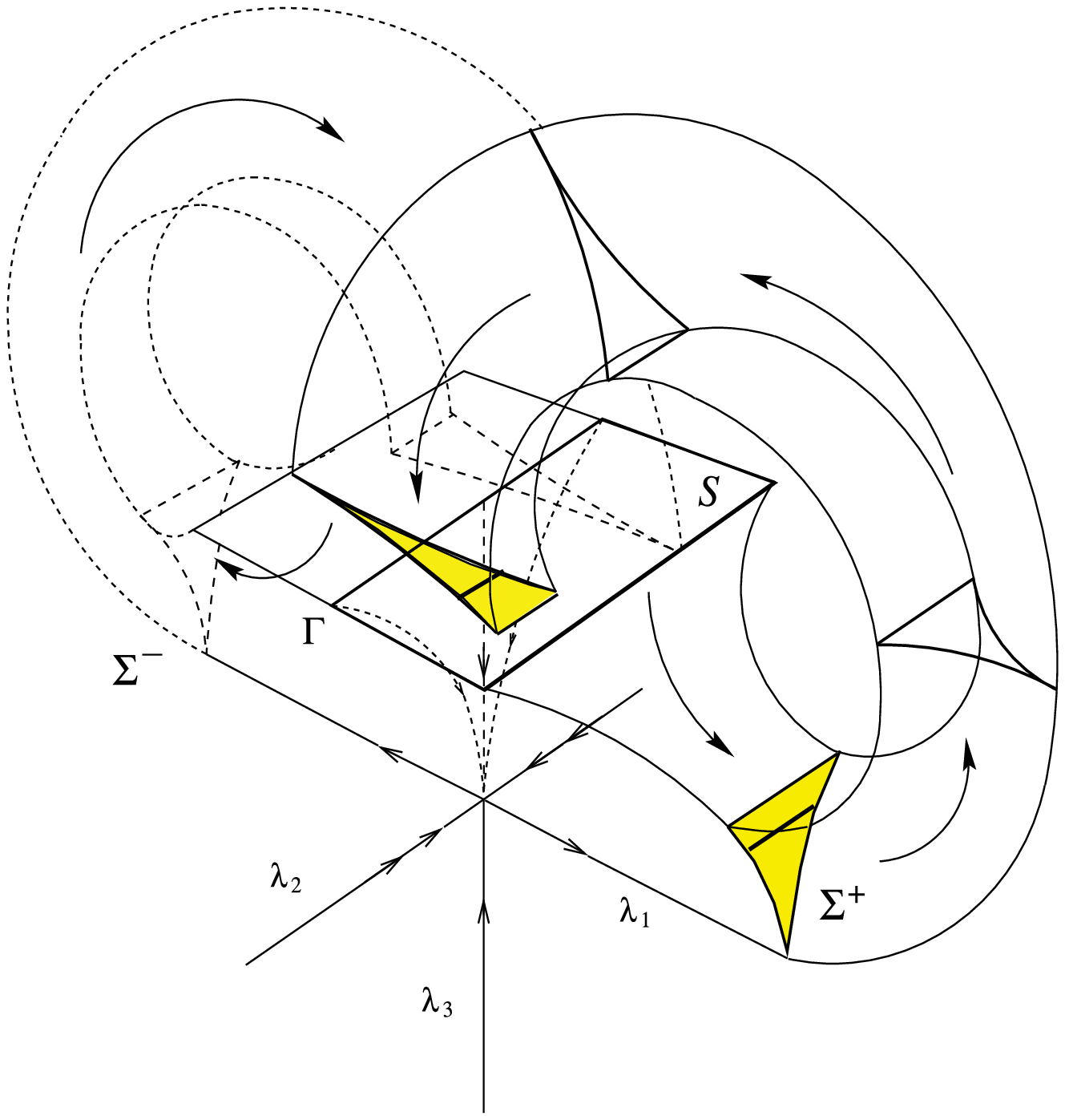}}

\caption{ Construction of the Geometric Lorenz flow}
\end{center}
\end{figure}

\noindent which was resulted of a tentative of modeling the
weather forecast in the years of sixty (1963). When the parameters
in \eqref{zz131} are $a = 10$, $r = 28$ and $b = 8/3$ then the
numeric simulation of this field exhibits a behavior which is
similar to the field $Y$ called Geometric Lorenz model, whose was
introduced by Guckenheimer (1976) \cite{rr4} and by Shilnikov
\cite{rr7}. To understand this geometric model, first consider the
flow $X_t$ associated to the Lorenz field near the origin $O=(0,
0, 0)$. Analogously, the field $Y$ has a hyperbolic singularity in
$O$ and, by Hartman-Grobman Theorem, it is conjugated to the
linearized equations in a neighborhood of the origin
$$ x'= \lambda_1 x , \hspace{.8cm}  y' = \lambda_2 y  ,\hspace{.8cm} z'=\lambda_3 z .$$
Solving this system with initial datas $(x(0), y(0), z(0)) = (x_0,
y_0, 1)$ we have:
$$x(t) = x_0(e^t)^{\lambda_1}, \hspace{.8cm}  y(t) =
y_0(e^t)^{\lambda_2}, \hspace{.8cm} z(t) = (e^t)^{\lambda_3}.$$

Fix $x_0 > 0$ and let $T$ be the positive time for which the orbit
intersects the plane $x = 1$, that is, $x(T) = 1$. Then $ e^T =
(x_0)^{{1} / {\lambda_1}}$ and so
$$ x(T) = 1,\hspace{.8cm}  y(T) = y_0(x_0)^{-\lambda_2 / {\lambda_1}}, \hspace{.8cm} z(T) =
(x_0)^{-\lambda_3 /{\lambda_1}}.  $$

Let $S = \{(x, y, 1) : |x| \leq 1/2, |y| \leq 1/2\}$ be a
transversal section to the fiel $Y$ such that the first return map
$P$ be defined in $S^{\ast} = S \setminus \{x = 0\}$. The line $x
= 0$ in $S$ is contained in the intersection between $W^s(0, Y )$
and $S$. Let
$$P : S^{\ast}\rightarrow int(S): \,\, p\mapsto P(p)$$
be defined by $P(p) = Y_{\tau} (p)$, where $\tau$ is the first
positive time such that $Y_{\tau} (p) \in S$. Assume the following
hypothesis over the field $Y$ (for more details see [GH90], p.
273):
\\

\noindent (h1) The point $O$ has eigenvalues $\lambda_1$,
$\lambda_2$, $\lambda_3$ such that $0 < -\lambda_3 < \lambda_1 <
-\lambda_2$, where $\lambda_3$ is the eigenvalue of the $z$-axis,
which is supposed to be $Y_t$-invariant.
\\

\noindent(h2) There exists a foliation $\F^s$ of $S$ whose
vertical leaves are such that: if $L \in \F^s$ and $P$ is defined
in $L$, then $P(L)$ stays contained in a leaf of $\F^s$. The
foliation $\F^s$ is part of the strong stable manifold of the flow
in the attractor which can be extended to a neighborhood of the
attractor \cite{rob}.
\\

\noindent(h3) Every point of $S^\ast$ returns to $S$, and the
return map $P$ is enough expansive in the direction which is
transverse to the leaves of $\F^s$.
\\

\noindent(h4) The flow is symmetric with respect to the rotation
$\theta= \pi$ around the $y$-axis.
\\

These four hypothesis define the \emph{Geometric Lorenz Flow}.
Analytically, these hypothesis may be reformulated by a coordinate
system $(x,y)$ over $S$ such that $P$ has the following
properties:
\\

\noindent(P1) The leaves of $\F^s$ are given by $x = c$, with
$-1/2\leq x \leq 1/2$.
\\

\noindent(P2) There are functions $f$ and $g$ such that $P$ has
the form
$$P(x, y) = (f(x), g(x, y))$$

for $x \neq 0$ and $P(-x,-y) = -P(x, y)$.
\\

\noindent(P3) $f'(x) \geq \lambda > \sqrt{2}$, for all $x \neq 0$
and $\lim_{x \rightarrow 0} f'(x) = \infty$.
\\

\noindent(P4) $0 < \dfrac{\partial g}{\partial y }<\delta<1$, for
all $x\neq 0$ e $\lim_{x \rightarrow 0} \dfrac{\partial
g}{\partial y }=0 $


\subsection{Foliations}

In the text, the leaves of the foliation $\F^s$, of hypothesis
$(h2)$, will be identified to the lines in $S$ whose
$x$-coordinate is fixed. For simplicity we will denote such a
foliation by $\F$ and its leaves by

\begin{equation*}
    \F_{x_0}=\left\{ (x_0,y,1)\in \mathbb{R}^3 \, | \, y\in \left[-\frac{1}{2}, \frac{1}{2}\right]
    \right\}.
\end{equation*}
Additionally, we will denote by
\begin{equation*}
    L^-=\F_{-1/2}, \hspace{2cm} L_0 = \F_0,
    \hspace{2cm}L^+=\F_{1/2}.
\end{equation*}
We call $L_0$ of \textit{singular leaf} and denote
$\F^*=\F\setminus L_0$. The one-dimensional map $f$, (P2), induces
a map $\hat{f}$ given by
\begin{align*}
    \hat{f}:\F^*& \longrightarrow \F \\
    \F_x & \longmapsto \F_{f(x)}.
\end{align*}
Whenever there is no ambiguity we shall denote $\hat{f}$ by $f$.

We define an \textit{order relation} on $\F$ by:
\begin{equation*}
    \F_x\leq \F_y \Leftrightarrow x \leq y
\end{equation*}
where $x, y \in [-1/2, 1/2]$.

\section{Proof of Theorem \ref{zz9}}\label{zz52}

The proof of Theorem \ref{zz9} will be developed in the next three
subsections. In \S \, \ref{zz113}, we define a pseudo orbit. In \S
\, \ref{zz114}, we prove some technical lemmas. Lastly, in \S \,
\ref{zz115}, we prove the theorem itself.

From now on, $(\Lambda,\varphi)$ is a geometric Lorenz flow, $P:
S^* \to S$ is the associated Poincar\'{e} map and $f: \F^* \to \F$
is the map associated to the foliation $\F$ over $S$.

\subsection{The pseudo orbit}\label{zz113}

This section concerns to the description of the pseudo orbit to be
used in the proof of the Theorem \ref{zz9}. We will develop it
under the assumption that
\begin{equation*}
    f(L^-)>L^- \hspace{.3cm}\textmd{and}\hspace{.3cm} f(L^+)=
    L^+
\end{equation*}
(recall the order relation on the set of leaves of $S$). The
remainder case, which is $f(L^-)\geq L^-$ and $f(L^+)<L^+$, will
be commented in Remark \ref{zz111} (\S \, \ref{zz112}).

Recall that a pseudo orbit is a bi-sequence (when dealing with
flows) composed by points and times. We are going to construct
separately two sequences. Firstly the sequence of points and then
the sequence of times. Finally we will argue that such a sequence
is, in fact, a pseudo orbit in the sense of the three kinds of
pseudo orbit cited in \S \, \ref{zz61}.

\subsubsection{Sequence of points}\label{zz62}

Denote by $V$ the set contained in the stable manifold of the
singularity $\sigma$ =(0,0,0) and ``under" the singular leaf $L_0$
of $S$ (see Figure \ref{zz122}), namely
\begin{equation*}
    V= \bigcup_{x\in L_0}\varphi(\mathbb{R}^+,x).
\end{equation*}
The positive orbit of points in $V$ converges to $\sigma$, thus
they do not cross the section $S$. In turns, all the others points
cross $S$ in the future. Now, for any $x\in U\setminus V$ let
$\tau(x)$ be the time spent by the flow to intersect $S$,
\begin{equation*}
    \tau(x)=\min\{t\in \mathbb{R}^+_* \, | \, \varphi(t,x)\in S \},
\end{equation*}
and $\pi(x)$ be such an intersection point
\begin{equation}\label{zz66}
    \pi(x)=\varphi(\tau(x),x).
\end{equation}

Put
\begin{equation*}
    \pi^0(x)=x \hspace{.3cm}\textmd{and}\hspace{.3cm} \pi^n(x)=\pi(\pi^{n-1}(x))\hspace{.4cm} \forall \, n>0.
\end{equation*}
Fix a constant $\Gamma$ much smaller than the distance between the
lateral leaves:
\begin{equation}\label{zz120}
    \Gamma \ll d(L^-,L^+)
\end{equation}
(here $d(A,B)=\inf\{d(x,y)\,|\,x\in A$ and $y\in B \}$ for any
sets $A$ and $B$). We will construct a one-sided sequence
$(x_n)_{n\in \mathbb{N}}$ in four steps. In the construction we
assume that $k$ is an arbitrary natural number or zero.

\vspace{.2cm} \noindent\textit{Step 1.} \vspace{.05cm}

The terms of the sequence of type $x_{4k}$ are points in
$W^{u,-}(\sigma)$
    which verify
    \begin{equation*}
        d(x_{4k},\sigma)=\frac{\Gamma}{\sqrt{2}.2^{k}}.
    \end{equation*}
(see Figure \ref{zz121})

\vspace{.2cm} \noindent\textit{Step 2.} \vspace{.05cm}

In the terms of type $x_{4k+1}$ we will impose two conditions.
    From the above item, we have defined the term $x_0$ and we also have
    $\pi(x_{4k})=\pi(x_0)$. So, we take $x_{4k+1}\in S^+$
    satisfying a first condition
    \begin{equation}\label{zz67}
        d(x_{4k+1},\pi(x_0))=\frac{\Gamma}{2^{k}}.
    \end{equation}
    Now, the positive orbit of any point in the set
    $U_{i=1}^{\infty}f^{-i}(L_0)$ intersects $L_0$. In turn
    such a set is dense in $S$. Based on this fact, we impose a second condition on $x_{4k+1}$:
    \begin{equation*}
        \varphi(s^1_k,x_{4k+1})\in L_0 \hspace{.3cm}\textmd{for}\hspace{.166cm}
        \textmd{some}\hspace{.3cm} s^1_k>1.
    \end{equation*}
    Note that, naturally, we have
    \begin{equation}\label{zz110}
        \F_{x_{4k+1}}<L^+=\F_{\pi(x_{0})}.
    \end{equation}

\vspace{.2cm} \noindent\textit{Step 3.} \vspace{.05cm}

Similar as in Step 1, we take $x_{4k+2}\in W^{u,+}(\sigma)$
verifying
    \begin{equation}\label{zz70}
        d(x_{4k+2},\sigma)=\frac{\Gamma}{\sqrt{2}.2^{k}}.
    \end{equation}

\vspace{.2cm} \noindent\textit{Step 4.} \vspace{.05cm}

Finally, the terms of type $x_{4k+3}$ belong to $S^-$ and we will
demand three
    conditions to them. Before this, note that
    $\pi(x_{4k+2})=\pi(x_{2})$ (and thus
    $\pi^2(x_{4k+2})=\pi^2(x_{2}))$ and we can reduce $\Gamma$
    such that $d(\pi(L^{-}),L^{-}) > \frac{\Gamma}{2^{k}}$.
    So, the first two conditions are:
    \begin{equation*}
        d(x_{4k+3},\pi^2(x_2))=\frac{\Gamma}{2^{k}}
    \end{equation*}
and
    \begin{equation*}
        \varphi(s^3_k,x_{4k+3})\in L_0 \hspace{.2cm}\textmd{for}\hspace{.2cm}
         \textmd{some}\hspace{.2cm}s^3_k>1.
    \end{equation*}
    In addition, note that $\F_{\pi(x_{4k+2})}=L^-$ and
    $\F_{\pi^2(x_{4k+2})}=f(L^-)$. Moreover we know that
    $f(L^-)>L^-$. These facts allow us to choose $x_{4k+3}\in S^-$
    satisfying the third condition:
    \begin{equation}\label{zz43}
        \F_{\pi(x_{4k+2})}=L^-<\F_{x_{4k+3}}<f(L^-)=\F_{\pi^2(x_{4k+2})}.
    \end{equation}

Now we extend the above sequence to a two-sided sequence
$(x_n)_{n\in \mathbb{Z}}$ by setting
\begin{equation}\label{zz75}
    x_{4k+i}=x_{4(-k-1)+i},
\end{equation}
for $k\leq -1$ and $i\in \{0, \dots , 3\}$. Moreover, the
expressions
\begin{equation*}
    \frac{\Gamma}{2^k} \hspace{.2cm}\textmd{and}\hspace{.2cm} \frac{\Gamma}{\sqrt{2}.2^k}
\end{equation*}
in the above equalities must be replaced respectively by
\begin{equation}\label{zz128}
    \frac{\Gamma}{2^{(-k)}} \hspace{.2cm}\textmd{and}\hspace{.2cm}
    \frac{\Gamma}{\sqrt{2}.2^{(-k)}}.
\end{equation}
This ends the construction of the sequence of points.

\subsubsection{Sequence of times}\label{zz63}

In the same way of the last subsection we will construct a
sequence of times $(t_n)_{n\in \mathbb{N}}$ in four steps. In the
construction we assume that $k$ is an arbitrary natural number or
zero.

\vspace{.2cm} \noindent\textit{Step 1.} \vspace{.05cm}

The terms of type $t_{4k}$ are the times that the
    points $x_{4k}$ spend to reach $\pi(x_{4k})\in S$, namely:
    \begin{equation}\label{zz65}
        t_{4k}=\tau(x_{4k}).
    \end{equation}

\vspace{.2cm} \noindent\textit{Step 2.} \vspace{.05cm}

Recall that $x_{4k+1}$ spends a time $s_k^1$ (through the flow) to
    reach the point $\varphi(s^1_k,x_{4k+1})$ in the singular leaf $L_0$. In turn, such a point
    tends to the singularity. So there exists a time
    $\tilde{s}_k^1$ verifying
    \begin{equation}\label{zz69}
        d(\varphi(s^1_k+\tilde{s}^1_k,x_{4k+1}),\sigma)=\frac{\Gamma}{\sqrt{2}.2^{k}}.
    \end{equation}
    Therefore we define the terms of
    type $t_{4k+1}$ as
    \begin{equation}\label{zz68}
        t_{4k+1}=s^1_k + \tilde{s}^1_k.
    \end{equation}

\vspace{.2cm} \noindent\textit{Step 3.} \vspace{.05cm}

The terms of type $t_{4k+2}$ are the times that the
    points $x_{4k+2}$ spend to reach $\pi^2(x_{4k+2})\in S$, namely:
    \begin{equation*}
        t_{4k+2}=\tau(x_{4k+2})+\tau(\pi(x_{4k+2})).
    \end{equation*}

\vspace{.2cm} \noindent\textit{Step 4.} \vspace{.05cm}

As in item (iii), $s^3_k$ is a time such that
    $\varphi(s^3_k,x_{4k+3})$ lies in the singular leaf $L_0$. So there
    is $\tilde{s}^3_k$ satisfying
    \begin{equation*}
        d(\varphi(s^3_k+\tilde{s}^3_k,x_{4k+3}),\sigma)=\frac{\Gamma}{\sqrt{2}.2^{k}}.
    \end{equation*}
    Therefore we define the terms of
    type $t_{4k+3}$ as
    \begin{equation*}
        t_{4k+3}=s^3_k + \tilde{s}^3_k.
    \end{equation*}

Now we extend the above sequence to a two-sided sequence
$(t_n)_{n\in \mathbb{Z}}$ by setting
\begin{equation}\label{zz76}
    t_{4k+i}=t_{4(-k-1)+i},
\end{equation}
for $k\leq -1$ and $i\in \{0, \dots , 3\}$. As in \eqref{zz128},
we replace $k$ by $-k$. This ends the construction of the sequence
of times.

\subsubsection{The pseudo orbit}\label{zz112}

\begin{definition}\label{zz82}
Let
\begin{equation}\label{zz109}
    (x_n,t_n)_{n\in \mathbb{Z}}
\end{equation}
be the bi-sequence whose component sequence $(x_n)_{n\in
\mathbb{Z}}$ is the sequence of points constructed in section
\ref{zz62} and the component sequence $(t_n)_{n\in \mathbb{Z}}$ is
the sequence of times constructed in section \ref{zz63}.
\end{definition}

\begin{figure}[h]
\begin{center}
\includegraphics[scale=0.67]{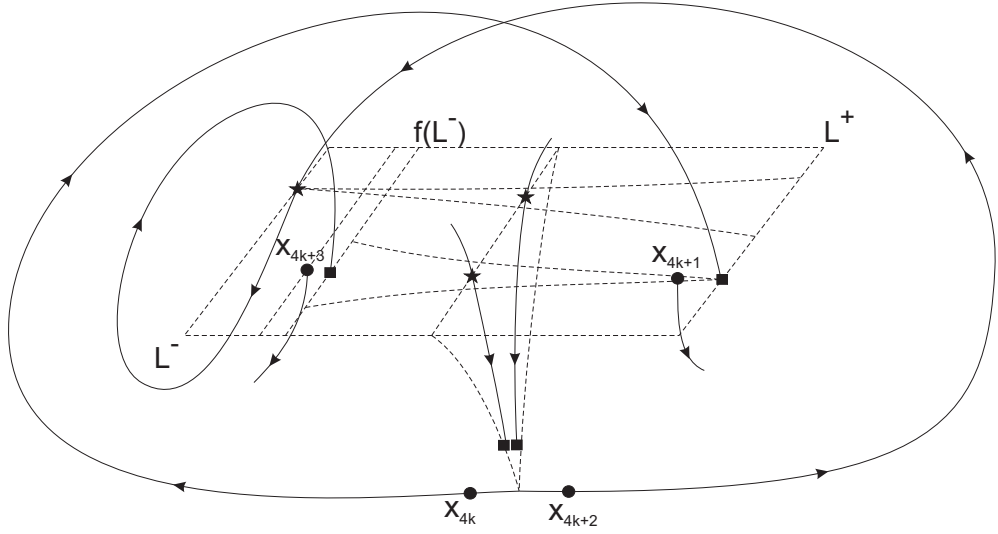}
\caption{\label{zz121} $k\geq 0$. The points in \textit{shape of
circle} compose the pseudo orbit. \textit{Square points}: in $L^+$
is $\pi(x_{4k})$, in $f(L^-)$ is $\pi^2(x_{4k+2})$, near the
origin at right is $\varphi(t_{4k+1},x_{4k+1})$ and at left is
$\varphi(t_{4k+3},x_{4k+3})$. \textit{Star points}: in the
singular leaf at right is $\varphi(s^1_k ,x_{4k+1})$, at left is
$\varphi(s^3_k,x_{4k+3})$ and in $L^-$ is $\pi(x_{4k+2})$.}
\end{center}
\end{figure}

\begin{remark}\label{zz111}
To the case
\begin{equation*}
    f(L^-)= L^- \hspace{.3cm}\textmd{and}\hspace{.3cm} f(L^+)<
    L^+
\end{equation*}
 we set another bi-sequence, say $(\overline{x}_n,\overline{t}_n)_{n\in
 \mathbb{Z}}$, whose construction is the same of \eqref{zz109}, up
 to demanding
 \begin{equation}\label{rrcc}
    \F_{\pi^2(x_{4k})}=f(L^+)<\F_{x_{4k+1}}< L^+=\F_{\pi(x_0)},
\end{equation}
instead of \eqref{zz110}, and
\begin{equation*}
    L^-=\F_{\pi(x_{4k+2})}<\F_{x_{4k+3}},
\end{equation*}
instead of \eqref{zz43} (the \textit{construction} of the sequence
of times does not change).

Additionally, to the case
\begin{equation*}\label{rrbb}
    f(L^-)> L^- \hspace{.3cm}\textmd{and}\hspace{.3cm} f(L^+)<
    L^+
\end{equation*}
we order as in \eqref{rrcc} and \eqref{zz43}.


The proofs of the further results are done only to the bi-sequence
$(x_n,t_n)_{n\in \mathbb{Z}}$. In Lemma \ref{zz102}, the analyzed
cases are simetricaly the same. The remainder facts (which deal
only with distances) are verbatim the same.

\end{remark}

\begin{proposition}\label{zz101}
The bi-sequence
\begin{equation}\label{zz42}
    (x_n,t_n)_{n\in \mathbb{Z}},
\end{equation}
given by Definition \ref{zz82}, is a $\delta$-average-pseudo orbit
for any $\delta>0$. Moreover, it is also a limit-pseudo orbit and
an asymptotic average-pseudo orbit.
\end{proposition}

\begin{proof}

Define a (positive) sequence $(A_m)_{m\in \mathbb{Z}}$ by
\begin{equation*}
    A_m=d(\varphi(t_m,x_m),x_{m+1}).
\end{equation*}
Firstly \textit{we claim that}
\begin{equation}\label{zz129}
    A_{4k+i}=\frac{\Gamma}{2^{|k|}},
\end{equation}
for any $k\in \mathbb{Z}$ and any $i\in\{0,1,2,3\}$ (see Figure
\ref{zz130}).

\begin{figure}[h]
\begin{center}
\includegraphics[scale=0.45]{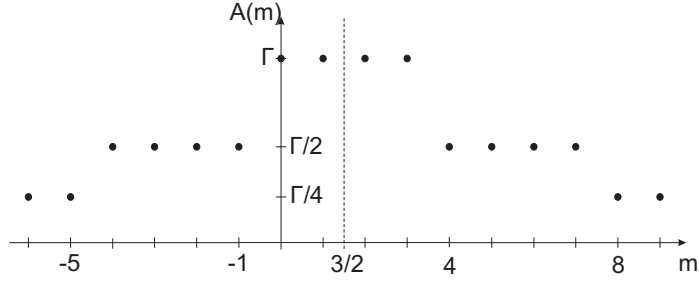}
\caption{\label{zz130} Graphic of the sequence $A_m$.}
\end{center}
\end{figure}

Indeed, suppose that $k\geq 0$. Then

\vspace{.2cm} \noindent\textit{(i)} \vspace{.05cm}
\begin{eqnarray*}
      d(\varphi(t_{4k},x_{4k}),x_{4k+1}) &\overset{\eqref{zz65}}{=}& d(\varphi(\tau(x_{4k}),x_{4k}),x_{4k+1})\notag \\
      \, &\overset{\eqref{zz66}}{=}& d(\pi(x_{4k}),x_{4k+1})\overset{\eqref{zz67}}{=}\frac{\Gamma}{2^k}.
\end{eqnarray*}

\vspace{.2cm} \noindent\textit{(ii)} \vspace{.05cm}

\begin{eqnarray*}
      d(\varphi(t_{4k+1},x_{4k+1}),x_{4k+2}) &=& \sqrt{\left[d(\varphi(t_{4k+1},x_{4k+1}),\sigma)\right]^2+[d(\sigma,x_{4k+2})]^2} \\
      \, &\overset{\eqref{zz68}}{=}& \sqrt{[d(\varphi(s_k^1+\tilde{s}_k^1,x_{4k+1}),\sigma)]^2+[d(\sigma,x_{4k+2})]^2} \\
      \, &\overset{\eqref{zz69},\eqref{zz70}}{=}& \sqrt{\left[\frac{\Gamma}{\sqrt{2}.2^k}\right]^2+
      \left[\frac{\Gamma}{\sqrt{2}.2^k}\right]^2}=\frac{\Gamma}{2^k}. \\
\end{eqnarray*}
The cases whose subscript index are $4k+2$ and $4k+3$ are
analogous to (i) and (ii) respectively. Summarizing we have
\begin{equation*}\label{zz125}
    d(\varphi(t_{4k+i},x_{4k+i}),x_{4k+i}=
    \frac{\Gamma}{2^{k}},
\end{equation*}
for $k\geq 0$ and $i\in \{0,1,2,3\}$. By the other hand, following
the same procedure, one can use the conversion formulas
\eqref{zz75}, \eqref{zz128} and \eqref{zz76}, to prove that
\begin{equation*}\label{zz126}
    d(\varphi(t_{4k+i},x_{4k+i}),x_{4k+i})=
    \frac{\Gamma}{2^{(-k)}},
\end{equation*}
for $k\leq -1$ and $i\in \{0,1,2,3\}$. This proves our claim.

From \eqref{zz129} we conclude that:

\vspace{.2cm} \noindent\textit{(a)} \vspace{.05cm}
\begin{equation*}
    \lim_{|m|\to \infty}A_m = 0.
\end{equation*}
that is, the sequence $(x_n,t_n)_{n\in \mathbb{Z}}$ is a
\emph{limit-pseudo orbit}.




Now we will verify that the sequence $(x_n,t_n)_{n\in \mathbb{Z}}$
is an \emph{asymptotic average-pseudo orbit}. In fact,

\begin{eqnarray*}
  \sum_{j=-n}^{n}d(\varphi(t_{j},x_{j}),x_{j+1}) \leq
  \sum_{j=0}^{2n}d(\varphi(t_{j},x_{j}),x_{j+1})  \leq
  4\sum_{j=0}^{2n}d(\varphi(t_{4j},x_{4j}),x_{4j+1}) \leq
  4\sum_{j=0}^{2n} \dfrac{\Gamma}{2^j}.
\end{eqnarray*}
Therefore
\begin{equation} \label{chocho}
    \lim_{n\to \infty}\dfrac{1}{n}\sum_{j=-n}^{n}d(\varphi(t_{j},x_{j}),x_{j+1}) = 0.
\end{equation}

Finally we are going to show that the bi-sequence $(x_n,t_n)_{n\in
\mathbb{Z}}$ is a \emph{$\delta$-average-pseudo orbit} for any
$\delta>0$. Given $\delta >0$ then we have, by \eqref{chocho},
that there exists $N>0$ such that for all $ n
> N$ and all $ u \in Z $ , we have:

\begin{eqnarray*}
  \frac{1}{n}\sum_{i=1}^{n}d(\varphi(t_{i+u},x_{i+u}),x_{i+u+1}) &\leq&
  \frac{1}{n}\sum_{i=-n}^{n}d(\varphi(t_{i},x_{i}),x_{i+1})< \delta.
\end{eqnarray*}

\end{proof}

\subsection{Technical lemmas}\label{zz114}

In this section we are going to prove two technical lemmas. Before
it, let us define an object and remaind some notations.

For any fixed point $y$ in $M$ the function $t\mapsto
\varphi(h(t),y)$ is continuous (because the flow is smooth and $h$
is a homeomorphism). Hence, from the compactness of $M$ we can set
the following definition.

\begin{definition}\label{zz103}
Let $h$ be a reparametrization in $Rep(\Delta)$. Then $\beta$ is a
positive real constant such that
\begin{equation}\label{zz85}
    d\left(\varphi(h(t),y),y\right)<\frac{\Gamma}{2}
\end{equation}
for all $t\in [-\beta,\beta]$ and all $y\in M$.
\end{definition}

Recall that if $(z_n,t_n)_{n\in \mathbb{Z}}$ is an arbitrary
bi-sequence then for any $n\in \mathbb{N}$ we denote by
\begin{equation*}
    s_n=\sum_{i=0}^{n-1}t_i \hspace{.4cm}\textmd{and}
    \hspace{.4cm}s_{-n}=\sum_{i=-n}^{-1}t_i.
\end{equation*}
Additionally, we set $s_0=0$.

\begin{lemma}\label{zz87}
Let $(x_n,t_n)_{n\in \mathbb{Z}}$ be the sequence \eqref{zz42}.
Take a point $y$ in $M$, a reparametrization $h$ in $Rep(\Delta)$
and an integer number $k$. If
\begin{equation}\label{zz83}
    d(\varphi(h(t),y),\varphi(t-s_{4k+u},x_{4k+u}))<2\Gamma
\end{equation}
for any $u\in \{0, \dots , 3\}$ and all $t\in
[s_{4k+u},s_{4k+u+1}-\beta)$ then
\begin{equation}\label{zz60}
    d(\varphi(h(t),y),\varphi(t-s_{4k+u},x_{4k+u}))<3\Gamma
\end{equation}
for any $u\in \{0, \dots , 3\}$ and all $t\in
[s_{4k+u},s_{4k+u+1})$.
\end{lemma}

\begin{proof}
Fix $u\in \{0, \dots ,3\}$. We are going to verify the inequality
\eqref{zz60} for any $t\in [s_{4k+u},s_{4k+u+1})$.

If
\begin{equation*}
    t\in[s_{4k+u},s_{4k+u+1}-\beta)
\end{equation*}
then we get \eqref{zz60} directly from \eqref{zz83}.

Now we will analyze the remainder case:
\begin{equation*}
    t\in [s_{4k+u+1}-\beta ,s_{4k+u+1}).
\end{equation*}
We will proceed by contradiction. Assume that there exist
$t_0\in[s_{4k+u+1}-\beta,s_{4k+u+1})$ for which the inequality
\eqref{zz60} does no hold, that is,
\begin{equation}\label{zz84}
    d(\varphi(h(t_0),y),\varphi(t_0-s_{4k+u},x_{4k+u})>3\Gamma.
\end{equation}
To shorter the expressions we will denote by
\begin{equation*}
    p(t)=\varphi(h(t),y)
\end{equation*}
and
\begin{equation*}
    q(t)=\varphi(t-s_{4k+u},x_{4k+u})
\end{equation*}
for $t\in \mathbb{R}$. Then
\begin{eqnarray*}
  d(p(t_0-\beta),q(t_0-\beta)) &>& d(q(t_0),p(t_0-\beta))-d(q(t_0),q(t_0-\beta)); \\
  \, &>& \left[d(p(t_0),q(t_0))-d(p(t_0),p(t_0-\beta))\right]-d(q(t_0),q(t_0-\beta)); \\
  \, &\overset{\eqref{zz84}}{>}& 3\Gamma -\left[d(p(t_0),p(t_0-\beta))+d(q(t_0),q(t_0-\beta))\right]; \\
  \, &\overset{\eqref{zz85}}{>}& 3\Gamma -
  \left(\frac{\Gamma}{2}+\frac{\Gamma}{2}\right)=2\Gamma.
\end{eqnarray*}
Summarizing, we have
\begin{equation*}
    d(\varphi(h(t_0-\beta),y)\varphi((t_0-\beta)-s_{4k+u},x_{4k+u}))>2\Gamma.
\end{equation*}
On the other hand, note that the time $t_0-\beta$ belongs to the
interval $[s_{4k+u},s_{4k+u+1}-\beta)$. This contradicts the
inequality \eqref{zz83} of our hypothesis. This contradiction ends
the proof.

\end{proof}

Given a real time $t$ then we can associate the following point in
$M$:
\begin{equation*}
    \varphi(t-s_n,x_n)
\end{equation*}
where $n$ is such that $t\in [s_n,s_{n+1})$. From now on, we will
say: \textit{the pseudo orbit} $(x_n,t_n)_{n\in \mathbb{Z}}$, as a
reference to all the points in $M$ obtained by the above
association. We will also write: \textit{the pseudo orbit}
$(x_n,t_n)_{n\in \mathbb{Z}}$ instead of \textit{the sequence}
$(x_n,t_n)_{n\in \mathbb{Z}}$.

If necessary, we can reduce the size of $\Gamma$ (see
\eqref{zz120}) in order to have
\begin{equation}\label{zz94}
    d(V,\Sigma^-)>3\Gamma \hspace{.4cm}\textmd{and}\hspace{.4cm} d(V,\Sigma^+)>3\Gamma.
\end{equation}
(see Figure \ref{zz122}). From the above inequalities and the
definition of the regions $U^+$ and $U^-$ we also get
\begin{equation}\label{zz91}
    d(\Sigma^+,U^-)>3\Gamma \hspace{.4cm}\textmd{and}\hspace{.4cm} d(\Sigma^-,U^+)>3\Gamma.
\end{equation}

\begin{figure}[h]
\begin{center}
\includegraphics[scale=0.5]{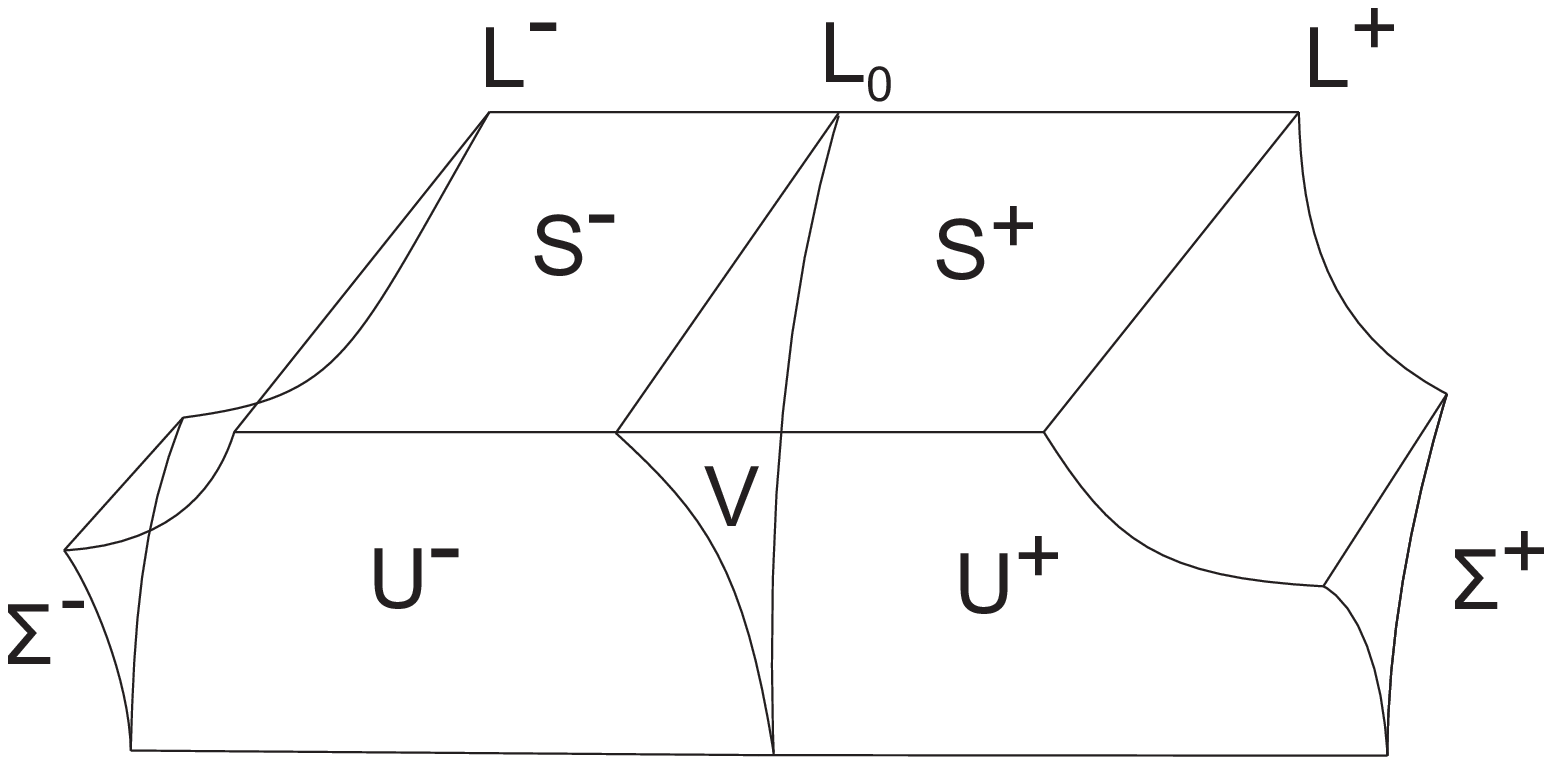}
\caption{\label{zz122} Regions}
\end{center}
\end{figure}

\begin{lemma}\label{zz102}
Let $(x_n,t_n)_{n\in \mathbb{Z}}$ be the pseudo orbit
\eqref{zz42}. Take a point $y$ in $M$, a reparametrization $h$ in
$Rep(\Delta)$. Then for each number $k\in \mathbb{Z}$ there exist
a number
\begin{equation*}
    u\in \{0, \dots ,4\}
\end{equation*}
and a time
\begin{equation*}
    t\in [s_{4k+u}, s_{4k+u+1}-\beta)
\end{equation*}
such that the following holds
\begin{equation}\label{zz86}
    d(\varphi(h(t),y),\varphi(t-s_{4k+u},x_{4k+u}))\geq2\Gamma.
\end{equation}
\end{lemma}

\begin{proof}
Fix $k\in \mathbb{Z}$. The proof will follow by exclusion: suppose
that there are no number $u\in \{0, \dots ,3\}$ and no time $t\in
[s_{4k+u},s_{4k+u+1}-\beta)$ for which they verify the inequality
\eqref{zz86}. Then we are going to exhibit a time $t_0$ inside the
interval $[s_{4k+4},s_{4k+5}-\beta)$ satisfying
\begin{equation}\label{zz100}
     d(\varphi(h(t_0),y),\varphi(t_0-s_{4k+4},x_{4k+4}))>2\Gamma.
\end{equation}
Such an exhibition concludes the proof of this lemma.

Before we start, let us recall some definitions: if $z$ is a point
in $M$ then $\tau(z)$ is the time spent by the flow through $z$ to
reach the cross section $S$; $\pi(z)$ is such an intersection
point; $\pi^n(z)=\pi(\pi^{n-1}(z))$ and $\F_{\pi(z)}$ is the leaf
of the foliation $\F$ on $S$ containing the point $\pi(z)$.

Firstly observe that, from our exclusion hypothesis at the
beginning and Lemma \ref{zz87}, we have
\begin{equation}\label{zz90}
    d(\varphi(h(t),y),\varphi(t-s_{4k+u},x_{4k+u}))<3\Gamma
\end{equation}
for any $u\in \{0, \dots , 3\}$ and all $t\in
[s_{4k+u},s_{4k+u+1})$. In other words, for each time $t$ in the
interval $[s_{4k},s_{4k+4})$ the point associated by the pseudo
orbit $(x_n,t_n)_{n\in \mathbb{Z}}$ and the point associated by
the (reparameterized) orbit through $y$ are at a distance less
than $3\Gamma$. \textit{Therefore, as the time in such an interval
goes forward, the orbit and the pseudo orbit intersects the cross
section $S$ the same number of times}. In particular, the points
$\pi(y)$ and $x_{4k+1}$ are close.

From definition of the terms of type $x_{4k+1}$, there exists a
number $m\in \mathbb{N}$ such that the point $\pi^m(x_{4k+1})$
lies on the singular leaf $L_0$. Recall that the cross section $S$
is a (disjoint) union of the leaf $L_0$ with the ``pieces" $S^-$
and $S^+$.

Now we suppose that there exists a number $j\in \{0, \dots ,m-1\}$
for which the points $\pi^j(\pi(y))$ and $\pi^j(x_{4k+1})$ are in
distinct pieces, namely:
\begin{equation}\label{zz88}
    \pi^j(\pi(y))\in S^{+}\Rightarrow \pi^j(x_{4k+1})\in
    S^{-}
\end{equation}
or
\begin{equation}\label{zz89}
    \pi^j(\pi(y))\in S^{-}\Rightarrow \pi^j(x_{4k+1})\in
    S^{+}.
\end{equation}
Let us proceed our argument using the implication \eqref{zz88}
(the usage of \eqref{zz89} would be symmetrical). In this case,
there exist a time $t_0\in [s_{4k+1},s_{4k+2})$ such that
\begin{itemize}
    \item either
\begin{equation*}
    \varphi(h(t_0),y)\in \Sigma^+ \hspace{.4cm}\textmd{and}
    \hspace{.4cm} \varphi(t_0-s_{4k+1},x_{4k+1})\in U^-
\end{equation*}
    \item or
\begin{equation*}
    \varphi(t_0-s_{4k+1},x_{4k+1})\in \Sigma^- \hspace{.4cm}\textmd{and}
    \hspace{.4cm}\varphi(h(t_0),y)\in U^+.
\end{equation*}
\end{itemize}
In the first case we have
\begin{equation*}
    d(\varphi(t_0),y),\varphi(t_0-s_{4k+1},x_{4k+1}))\geq
    d(\Sigma^+,U^-)\overset{\eqref{zz91}}{>}3\Gamma.
\end{equation*}
In the second case we have
\begin{equation*}
    d(\varphi(t_0),y),\varphi(t_0-s_{4k+1},x_{4k+1}))\geq
    d(U^+,\Sigma^-)\overset{\eqref{zz91}}{>}3\Gamma.
\end{equation*}
In both cases we fall in contradiction with \eqref{zz90}. So, this
contradiction means that the points $\pi^j(\pi(y))$ and
$\pi(^j(x_{4k+1})$ are always in the same piece, namely:
\begin{equation}\label{zz92}
    \pi^j(\pi(y))\in S^{\pm}\Rightarrow \pi^j(x_{4k+1})\in
    S^{\pm}
\end{equation}
for all $j\in \{0, \dots , m-1\}$.

 Now consider the leaves $\F_{\pi(y)}$ and $\F_{x_{4k+1}}$ of the
foliation $\F$ defined on $S$. Recall that we defined an order
relation ``$\leq$" on $\F$. We will finish the proof of this lemma
by analyzing such an order relation over these leaves,
$\F_{\pi(y)}$ and $\F_{x_{4k+1}}$.

\vspace{.5cm}
\noindent\textit{Case(1):\,$\,\F_{\pi(y)}>\F_{x_{4k+1}}$.}
\newline

Observe that the functions $\left. f \right|_{S^+}$ and $\left. f
\right|_{S^-}$ are increasing. This fact together with the
implication \eqref{zz92} and the hypothesis of this case imply
\begin{equation*}
    L_0=\F_{\pi^m(x_{4k+1})}<\F_{\pi^{m+1}(y)}
\end{equation*}
(see Figure \ref{zz123}).

\begin{figure}[h]
\begin{center}
\includegraphics[scale=0.6]{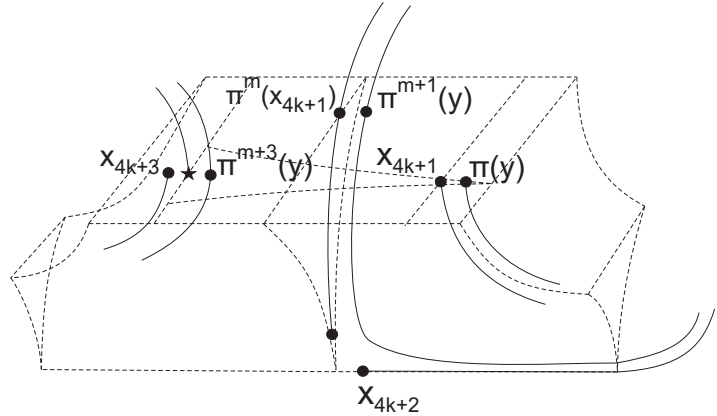}
\caption{\label{zz123} The orbit and the pseudo orbit still close
for times in $[s_{4k},s_{4k+4})$. The star point is
$\pi^2(x_{4k+2})$.}
\end{center}
\end{figure}

\noindent Thus
\begin{equation*}
    \F_{\pi(x_{4k+2})}=L^-<\F_{\pi^{m+2}(y)}<L_0.
\end{equation*}
Hence
\begin{equation}\label{zz97}
    \F_{x_{4k+3}}
    \overset{\eqref{zz43}}{<}\F_{\pi^2(x_{4k+2})}<\F_{\pi^{m+3}(y)}.
\end{equation}
From definition of the terms of type $x_{4k+3}$, there exists a
number $q\in \mathbb{N}$ such that the point $\pi^q(x_{4k+3})$
lies on the singular leaf $L_0$. From the same argument we used to
conclude the implication \eqref{zz92} (it is: the orbit and the
pseudo orbit must hit $S$ simultaneously in same pieces) we get
\begin{equation}\label{zz98}
    \F_{\pi^j(x_{4k+3})}\in S^{\pm} \Rightarrow \F_{\pi^{j+m+3}(y)}\in
    S^{\pm}
\end{equation}
for all $j\in \{0, \dots , q-1\}$. Since the functions $\left. f
\right|_{S^+}$ and $\left.f \right|_{S^-}$ are increasing then, by
\eqref{zz97} and \eqref{zz98}, we have
\begin{equation*}
    L_0=\F_{\pi^q(x_{4k+3})}<\F_{\pi^{q+m+3}(y)}
\end{equation*}
Therefore,
\begin{itemize}
    \item either
    \begin{equation}\label{zz99}
        \varphi(h(t_0),y)\in \Sigma^+ \hspace{.4cm}\textmd{and}\hspace{.4cm} \varphi(t_0-s_{4k+3},x_{4k+3})\in V
    \end{equation}
    for some $t_0\in[s_{4k+3},s_{4k+4})$
    \item or
    \begin{equation*}
        \varphi(h(t_0),y)\in \Sigma^+ \hspace{.4cm}\textmd{and}\hspace{.4cm} \varphi(t_0-s_{4k+3},x_{4k+3})\in
        U^-
    \end{equation*}
    for some $t_0\in[s_{4k+4},s_{4k+5})$ (see Figure \ref{zz124})
    \item or
    \begin{equation*}
        \varphi(t_0-s_{4k+3},x_{4k+3})\in \Sigma^- \hspace{.4cm}\textmd{and}\hspace{.4cm}
        \varphi(h(t_0),y)\in U^+
    \end{equation*}
    for some $t_0\in[s_{4k+4},s_{4k+5})$.
\end{itemize}

\begin{figure}[h]
\begin{center}
\includegraphics[scale=0.6]{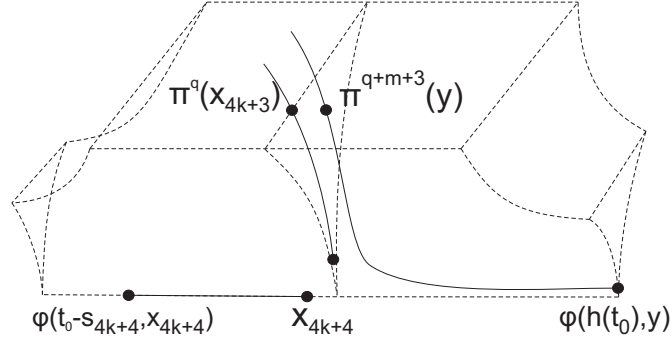}
\caption{\label{zz124} The orbit and the pseudo orbit gets far for
a time $t_0$ in $[s_{4k+4},s_{4k+5}-\beta)$.}
\end{center}
\end{figure}

The pertinence relations in \eqref{zz99} \textit{may not} occur
because, using \eqref{zz94}, we can contradict \eqref{zz90}. On
the other hand, the two remaining conditions means that
\begin{equation*}
    d(\varphi(t_0),y),\varphi(t_0-s_{4k+4},x_{4k+4}))\overset{\eqref{zz91}}{>}
    3\Gamma>2\Gamma
\end{equation*}
for some $t_0\in[s_{4k+4},s_{4k+5})$. This is almost what we
wanted to show (see \eqref{zz100}).

Recall the definition of $\beta$ on \eqref{zz85}. We conclude the
current case \textit{claiming that}
\begin{equation*}
    t_0\in[s_{4k+4},s_{4k+5}-\beta).
\end{equation*}
In fact, the pseudo orbit through $x_{4k+4}$ must cross
$\Sigma^-$. This implies that
\begin{equation*}
    \exists \, \eta\geq 0 ;
    \hspace{.3cm}\varphi((t_0+\eta)-s_{4k+4},x_{4k+4})\in \Sigma^-.
\end{equation*}
By the other side, we have
\begin{equation*}
    \varphi(s_{4k+5}-s_{4k+4},x_{4k+4})\in S^+.
\end{equation*}
Once $\overline{\Sigma^-}$ and $\overline{S^+}$ are disjoint then
the time
\begin{equation*}
    s_{4k+5}-(t_0+\eta)
\end{equation*}
is bounded away from zero (uniformly on $k$). Hence we can reduce
$\beta$, if necessary, to get
\begin{equation*}
    \beta< s_{4k+5}-(t_0+\eta).
\end{equation*}
So
\begin{equation*}
    t_0<s_{4k+5}-\beta.
\end{equation*}
As we wanted to show.

\vspace{.5cm}
\noindent\textit{Case(2):\,$\,\F_{\pi(y)}\leq\F_{x_{4k+1}}$.}
\newline

In this case we have
\begin{equation*}
    \F_{\pi^m(\pi(y))}\leq \F_{\pi^m(x_{4k+1})}=L_0.
\end{equation*}

Therefore the orbit either gets into $V$ and stays there or gets
into $U^-$ and ``escapes" through $\Sigma^-$. In turn, the pseudo
orbit gets into $V$ and escapes through $\Sigma^+$. Anyway, we can
use the same argument as in the \textit{Case(1)} to contradicts
the inequality \eqref{zz90}. This means that this current case
\textit{may not} occur. The proof of the lemma is over.

\end{proof}

\subsection{Conclusion}\label{zz115}

In this section we are going to prove the Theorem \ref{zz9}.

\vspace{.7cm} \noindent \textit{Proof of Theorem \ref{zz9}}.

Firstly we will prove that Lorenz flows $(\Lambda, \varphi)$ has
not the $\Delta$-asymptotic average shadowing property
($\Delta$-AASP).

Our argument will be by contradiction: suppose that $(\Lambda,
\varphi)$ has the $\Delta$-AASP. According to Definition
\ref{zz48}, for any asymptotic average-pseudo orbit
$(z_n,t_n)_{n\in \mathbb{Z}}$ there exist a point $y\in M$ and a
reparametrization $h\in Rep(\Delta)$ such that
\begin{equation}\label{zz107}
    \lim_{n\to\infty}\frac{1}{n}\sum_{i=0}^n\int_{s_i}^{s_{i+1}}d(\varphi({h(t)},y),
    \varphi({t-s_i},z_i))dt=0.
\end{equation}
Our target is to contradict this last sentence.

Let $(x_n,t_n)_{n\in \mathbb{Z}}$ be the asymptotic average-pseudo
orbit given by the Proposition \ref{zz101}. Take a point $y\in M$
and a reparametrization $h\in Rep(\Delta)$. The procedure is: we
will find a \textit{subsequence of} $(x_n,t_n)_{n\in \mathbb{Z}}$
such that the limit \eqref{zz107} will be \textit{strictly
positive}. This shall be enough to conclude the proof due to the
arbitrariness of the point $y$ and the reparametrization $h$.

By Lemma \ref{zz102}, there exist sequences $(u_k)_{k\in
\mathbb{Z}}$ and $(t_k)_{k\in \mathbb{Z}}$ verifying
\begin{equation*}\label{zz108}
    u_k\in \{0, \dots ,4\}
\end{equation*}
and
\begin{equation}\label{zz105}
    t_k\in [s_{4k+u_k},s_{4k+u_k+1}-\beta)
\end{equation}
for all $k\in \mathbb{Z}$ and such that the following holds:
\begin{equation}\label{zz104}
    d(\varphi(h(t_k),y),\varphi(t_k-s_{4k+u_k},x_{4k+u_k}))\geq 2\Gamma
\end{equation}
for all $k\in \mathbb{Z}$.

Now \textit{we claim that}
\begin{equation}\label{zz106}
    d(\varphi(h(t_k+r),y),\varphi((t_k+r)-s_{4k+u_k},x_{4k+u_k}))>\Gamma
\end{equation}
for all $r\in[-\beta,\beta]$.

In fact, one can verify this claim by using the equation
\eqref{zz104}, the definition of $\beta$ (Definition \ref{zz103})
and a same argument used in the proof of Lemma \ref{zz87}. In
other words, this claim says that: the points $\varphi(h(t_k),y)$
and $\varphi((t_k)-s_{4k+u_k},x_{4k+u_k})$, once $\Gamma$-far,
spend a time larger that $\beta$ to be $\Gamma$-close again (if
they do).

Therefore, consider the following subsequence of $(x_n,t_n)_{n\in
\mathbb{Z}}$:
\begin{equation*}
    (x_{4k+u_k},t_{4k+u_k})_{k\in \mathbb{Z}}.
\end{equation*}
Then
\begin{eqnarray*}
  \lim_{k\to \infty}\frac{1}{4k+u_k}\sum_{i=0}^{4k+u_k}\int_{s_i}^{s_{i+1}}
     d(\varphi(h(t),y),\varphi(t-s_i,x_i)dt &\geq& \, \\
  \lim_{k\to \infty}\frac{1}{5k}\sum_{i=1}^{k}\int_{s_{4i+u_i}}^{s_{4i+u_i+1}}
  d(\varphi(h(t),y),\varphi(t-s_{4i+u_i},x_{4i+u_i})dt &\overset{\eqref{zz105}}{\geq}& \, \\
  \lim_{k\to \infty}\frac{1}{5k}\sum_{i=1}^{k}\int_{t_i}^{t_i +\beta}
  d(\varphi(h(t),y),\varphi(t-s_{4i+u_i},x_{4i+u_i})dt &\overset{\eqref{zz106}}{\geq}& \, \\
  \lim_{k\to \infty}\frac{1}{5k}\sum_{i=1}^{k}\beta\Gamma &=&
  \frac{\beta\Gamma}{5}.
\end{eqnarray*}
As we wanted to show.

In this second part we are going to argue about Lorenz flows
having not the $\Delta$-average shadowing property ($\Delta$-ASP).
The proof is done also in the indirect method: suppose that
$(\Lambda , \varphi)$ has the $\Delta$-ASP. Then consider the same
pseudo orbit as before: $(x_n,t_n)_{n\in \mathbb{Z}}$. By
Proposition \ref{zz101}, it is an $\delta$-average-pseudo orbit
for every $\delta >0$. On the other hand
\begin{equation*}
    \lim_{n\to\infty}\frac{1}{n}\sum_{i=0}^n\int_{s_i}^{s_{i+1}}d(\varphi({h(t)},y),
    \varphi({t-s_i},x_i))dt> \frac{\beta \Gamma}{5}
\end{equation*}
for any $y\in M$ and any $h\in Rep(\Delta)$. So the Definition
\ref{zz46} fails for every $\epsilon < \frac{\beta \Gamma}{5}$
(the \textit{lim \hspace{-.08cm}sup} in \eqref{zz44} would be
larger than $\epsilon$). This finishes the second part.

Finally, one can show that Lorenz flows $(\Lambda , \varphi)$ has
not the $\Delta$-LSP.

\hfill $\Box$

\end{document}